\documentclass[11pt]{article}
\usepackage{amsmath,amssymb,amsthm, epsfig}
\usepackage{amsfonts}
\usepackage{amsmath}
\usepackage{amssymb}
\usepackage{color}
\usepackage[mathscr]{eucal}

\title{Universal moduli of continuity for solutions \\ to fully nonlinear elliptic equations}

\author{by  \vspace{0.3cm}  \\    \textsc{eduardo v. teixeira} \\ \textit{\footnotesize Universidade Federal do Cear\'a}  \\ \textit{\footnotesize Fortaleza, CE, Brazil} }

\date{}

\newlength{\hchng}
\newlength{\vchng}
\newtheorem{theorem}{Theorem}[section]
\newtheorem{lemma}[theorem]{Lemma}

\newtheorem{corollary}[theorem]{Corollary}

\theoremstyle{definition}

\theoremstyle{remark}
\newtheorem{remark}[theorem]{Remark}
\numberwithin{equation}{section}

\newcommand{\intav}[1]{\mathchoice {\mathop{\vrule width 6pt height 3 pt depth  -2.5pt
\kern -8pt \intop}\nolimits_{\kern -6pt#1}} {\mathop{\vrule width
5pt height 3  pt depth -2.6pt \kern -6pt \intop}\nolimits_{#1}}
{\mathop{\vrule width 5pt height 3 pt depth -2.6pt \kern -6pt
\intop}\nolimits_{#1}} {\mathop{\vrule width 5pt height 3 pt depth
-2.6pt \kern -6pt \intop}\nolimits_{#1}}}

\begin{document}
\maketitle

\begin{abstract}
This paper provides universal, optimal moduli of continuity for
viscosity solutions to fully nonlinear elliptic equations $F(X,
D^2u) = f(X)$, based on weakest integrability properties of $f$ in
different scenarios. The primary result established in this work
is a sharp Log-Lipschitz estimate on $u$ based on the $L^n$ norm
of $f$, which corresponds to optimal regularity bounds for the
critical threshold case. Optimal $C^{1,\alpha}$ regularity
estimates are delivered when $f\in L^{n+\epsilon}$. The limiting
upper borderline case, $f\in  L^\infty$, also has transcendental
importance to elliptic regularity theory and its applications. In
this paper we show, under convexity assumption on $F$, that $u \in
C^{1,\mathrm{Log-Lip}}$, provided $f$ has bounded mean
oscillation. Once more, such an estimate is optimal. For the lower
borderline integrability condition allowed by the theory, we
establish interior \textit{a priori} estimates on the
$C^{0,\frac{n-2\varepsilon}{n-\varepsilon}}$ norm of $u$ based on
the $L^{n-\varepsilon}$ norm of $f$, where $\varepsilon$ is the
Escauriaza universal constant. The exponent
$\frac{n-2\varepsilon}{n-\varepsilon}$ is optimal. When the source
function $f$ lies in $L^q$,  $n > q > n-\varepsilon$, we also
obtain the exact, improved sharp H\"older exponent of continuity.

\medskip


\medskip


\end{abstract}


\section{Introduction}

The key objective of the present paper is to derive, interior, optimal and universal moduli of continuity to solutions, as well their gradients, for second order equations of the form
$$
    F(X, D^2u) = f(X), \text{ in } Q_1 \subset \mathbb{R}^n,
$$
under appropriate conditions on the coefficients and integrability properties of $f$. Such information is crucial in a number of problems. This is particularly meaningful in certain geometric and free boundary problems governed by these equations, where one needs fine control of the growing rate of solutions from their level surfaces.

\par

Throughout this paper we shall work under uniform ellipticity condition on the operator $F\colon Q_1 \times \mathcal{S}(n) \to \mathbb{R}$, that is, there exist two positive constants $0 < \lambda \le \Lambda$ such that, for any $M \in \mathcal{S}(n)$ and $X \in Q_1$,
\begin{equation}  \label{unif ellip}
    \lambda \|P\| \le F(X, M+P) - F(X, M) \le \Lambda \|P\|, \quad \forall P \ge 0.
\end{equation}

Following classical terminology, any operator $F$ satisfying ellipticity condition \eqref{unif ellip} will be called $(\lambda, \Lambda)$-elliptic. Also, any constant or entity that depends only on dimension and the ellipticity parameters $\lambda$ and $\Lambda$ will be called \textit{universal}. For normalization purposes, we assume, with no loss of generality, throughout the text that $F(X, 0) = 0, \quad \forall X \in Q_1.$

Regularity theory for viscosity solutions to fully nonlinear elliptic equations has been a central subject of research since the fundamental work of Krylov and Safonov, \cite{KS1, KS2} on Harnack inequality for solutions to (linear) non-divergence form  elliptic equations unlocked the theory. By linearization, solutions to $F(D^2u) = 0$ as well as first derivatives of $u$, $u_\nu$, do satisfy  linear elliptic equations in non-divergence form with bounded measurable coefficients. Therefore, solutions are \textit{a priori} $C^{1,\alpha}$ for some unknown $\alpha> 0$. The language of viscosity solutions allows the same conclusion without linearizing the equation, see \cite{CC}.

\par

Under concavity or convexity assumption on $F$, Evans and Krylov, \cite{E} and \cite{K1, K2}, obtained $C^{2,\alpha}$ estimates for solutions to $F(D^2u) = 0$, establishing, therefore, a classical existence theory for such equations.  Two and a half decades passed until Nadirashvili and Vladut, \cite{NV1, NV2} exhibit solutions to uniform elliptic equations whose Hessian blow-up, i.e., that are not $C^{1,1}$.  More dramatically, they have recently shown that given any $0<\tau < 1$ it is possible to build up a uniformly elliptic operator $F$, whose solutions are not $C^{1,\tau}$, see \cite{NV}, Theorem 1.1. A major open question in the theory has been to determine, for a given $F$, the optimal universal $\alpha_F$ for which there are $C^{1,\alpha_F}$ estimates for solutions to the homogeneous equation $F(D^2 u) = 0$.

\par

The corresponding regularity theory for heterogeneous, non-constant coefficient equations $F(X, D^2u) = f(X)$ is rather more delicate and, clearly, solutions may not be as regular as solutions to the homogeneous, constant-coefficient equation. In this setting, L. Caffarelli, in a groundbreaking work, \cite{C1}, established $W^{2,p}$,  \textit{a priori} estimates for solutions to
$F(X, D^2u) = f\in L^p,$ $p > n$, under some sort of VMO condition on the coefficients, provided the  homogeneous, constant-coefficient equation
$F(X_0, D^2 u) = 0$ has  $C^{1,1}$ \textit{a priori}  estimates. In \cite{C1} it is also shown that for $p < n$, there exists a uniformly elliptic operator satisfying the hypothesis of Caffarelli's Theorem, for which $W^{2,p}$ estimates fail. Nevertheless, for a fixed elliptic operator $F(X, D^2u)$,   L. Escauriaza, \cite{Es}, established the existence of a universal constant $\varepsilon = \varepsilon(n, \Lambda,\lambda) \in (0, \frac{n}{2})$, for which Caffarelli's $W^{2,p}$ theory is valid for $p \ge n - \varepsilon$. Hereafter in this paper, $\varepsilon = \varepsilon (n, \Lambda,\lambda)$ will always denote the universal Escauriaza constant. As a product  of \cite{Es}, the following Harnack type inequality holds for nonnegative solutions:
$$
    \sup\limits_{B_{r/2}} u \le C \left [ \inf\limits_{B_{r/2}} + r^{2 - \frac{n}{n-\varepsilon}} \|f\|_{L^{n-\varepsilon}(B_r)} \right ].
$$
In particular, solutions are locally $C^\delta$ for some $\delta > 0$; however, the optimal $\delta > 0$ cannot be determined through oscillation decay coming from Harnack inequality.  A. Swiech, in an important paper, \cite{S}, obtained $W^{1,q}$ \textit{almost optimal} regularity estimates, in the range $q<p*$, for every $n-\epsilon<p\leq n$. 

\par

Nevertheless, in many situations, it is important to obtain precisely the optimal regularity available for solutions. For instance if one is required to obtain fine control on how fast solutions $u$ grow from, say, its zero level set. In this setting, knowing exactly the optimal universal modulus of continuity of $u$ is essential.  The main goal of this present article is, therefore, to establish the best modulus of continuity available for solutions to heterogeneous fully nonlinear elliptic equations
$$
    F(X, D^2u) = f,
$$
under minimal and borderline  integrability properties of $f$. We will show, see Section \ref{Calpha theory}, that if $f \in L^{n-\varepsilon}$, then solutions are locally ${C^{0, \frac{n-2\varepsilon}{n-\varepsilon}}}$ and such a continuity is optimal. Differently from \cite{C1} and \cite{E}, we do not assume constant coefficient equation has \textit{a priori} $C^{1,1}$ estimates. Instead, our approach requires only smallness assumption on the oscillation of the coefficients of $F$, see \eqref{beta}.

\par

The threshold case $f\in L^n$ has remarkable importance to the theory as it appears in many analytic and geometric tools such as Alexandrov- Bakel'man-Pucci estimates and the classical Harnack inequality. Indeed, condition $f \in L^n$ divides the regularity theory from continuity estimates to differentiability properties of solutions. For instance, in the analysis of the classical Poisson equation, $\Delta u = f$, the derivative of the fundamental solution, $D\Gamma$, behaves like $|Z|^{1-n}$. Thus $D\Gamma \in L^{\frac{n}{n-1} - \delta} $, for any $\delta > 0$, but $D\Gamma \not \in L^{\frac{n}{n-1}}$. In conclusion, as  $n$ is the dual exponent of $\frac{n}{n-1}$, we find out that if $f \in L^{n + \delta}$, solutions are Lipschitz continuous (and in fact differentiable), but it is impossible to bound the Lipschitz norm of $u$ by the $L^n$ norm of $f$.

\par

For the critical case, $f \in L^n$, we show, see Theorem
\ref{Log-Lip theory},  that solutions to $F(X, D^2u) = f$ are
Log-Lipschitz continuous, that is,
$$
    |u(X) - u(Y)| \lesssim- \log \left ( |X-Y| \right ) \cdot |X-Y|.
$$
Such a Theorem can be understood as qualitative improvement to the
fact that $u \in C^\alpha$ for any $\alpha < 1$, and, as far as we
know, Theorem \ref{Log-Lip theory} is the first optimal regularity
estimate for the case when the source function is in the critical
threshold space $L^n$. Simple adjustments on the method developed in the proof of  Theorem
\ref{Log-Lip theory} further yield $\nabla u \in \text{VMO}$, provided $f\in L^n$ and the coefficients of $F$ are in $\text{VMO}$.

In accordance to  \cite{C1}, for sources in $f \in L^q$, $q > n$,
solutions are $C^{1,\alpha}$. Our approach  recovers such an
estimate and revels explicitly the optimal $\alpha$ allowed. The
upper borderline case, $f \in L^\infty$, or $f \in \text{BMO}$, has also been another
intriguing puzzle, with major importance to the (fully nonlinear)
elliptic regularity theory and its applications. In this article
we show that for equations with $C^{2, \epsilon}$ \textit{a
priori} estimates for $F$-harmonic functions, solutions are
$C^{1,\mathrm{Log-Lip}}$ smooth, in the sense that
$$
    \left | u(X) - \left [ u(Y) + \nabla u(Y) \cdot X \right ] \right| \lesssim r^2 \log
    r^{-1}, \quad r= |X-Y|.
$$

The results presented in this paper are new even when projected to
the context of (non-divergence) linear elliptic equations. They
are in accordance to the rich regularity theory, established via
singular integrals and potential analysis, for divergence form
equations of Poisson's type, $-\Delta u =f$.  The methods employed
in this present work, on the other hand, are inspired by the
primary essences of the revolutionary, seminal ideas from
\cite{C1}, and they allow further generalizations accordantly, in particular with connections to estimates from geometric measure theory, see \cite{M}.

\par

We finish up this introduction by pointing out that, as in
\cite{C1}, our regularity estimates are not committed to the
notion of viscosity weak solution one chooses to use, $C^2$,
$C^{1,1}$, $W^{2,p}$, etc. To avoid possible technical
inconsistence to the definition of viscosity solution chosen, one
can interpret Theorem \ref{Calpha theory}, Theorem \ref{Log-Lip
theory}, Theorem \ref{C1,Log-Lip theory with BMO} and their
consequences as \textit{a priori} estimates, that depend only on
${L^{q}}$ norm of $f$, $n-\varepsilon \le q \le +\infty$.

\section{Optimal $C^{0,\alpha}$ regularity estimates} \label{Calpha theory}

Following \cite{C1}, for fixed $X_0 \in Q_1$, we measure the oscillation of the coefficients of $F$ around $X_0$ by
\begin{equation} \label{beta}
    \beta(X_0, X) := \sup\limits_{M \in \mathcal{S}(n) \setminus \{0\}} \dfrac{|F(X, M) - F(X_0, M)|}{\|M\|}.
\end{equation}
For simplicity we will write
$$
    \beta(0, X)  =: \beta(X).
$$

Our strategy  for proving optimal $C^{0,\alpha}$ regularity estimates is based on a refined compactness method, see \cite{C1}. It relies on a fine control of decay of oscillation based on the regularity theory available for a nice limiting equation. Next lemma is the key access point for our approach.

\begin{lemma}[Compactness Lemma]\label{compactness} Let $u$ be a viscosity solution to $F(X, D^2u) = f(X)$, with $|u| \le 1$,  in $Q_1$. Given any $\delta > 0$, we can find $\eta >0$ depending only on $n$, $\Lambda,~\lambda$ and $\delta$, such that if
\begin{equation}\label{Cond comp lemma}
    \intav{Q_1} \beta(X)^n dX \le \eta^n \quad \text{and} \quad \int_{Q_1} |f(X)|^{n-\varepsilon} dX \le \eta^{n-\varepsilon},
\end{equation}
then we can find a function $h \colon Q_{1/2} \to \mathbb{R}$ and a $(\lambda, \Lambda)$-elliptic, constant coefficient operator $\mathfrak{F} \colon \mathcal{S}(n) \to \mathbb{R}$, such that
\begin{equation}\label{eq comp lemma}
    \mathfrak{F}(D^2h) = 0, \text {in } Q_{1/2},
\end{equation}
in the viscosity sense and
\begin{equation}\label{thesis comp lemma}
    \sup\limits_{Q_{1/2}} |u - h| \le \delta.
\end{equation}

\end{lemma}
\begin{proof} Let us suppose, for the purpose of contradiction, that there exists a $\delta_0 > 0$ for which the thesis of the Lemma fails. That means that we could find a sequence of functions $u_j$, $|u_j| \le 1$ in $Q_1$,  a sequence of $(\lambda, \Lambda)$-elliptic operators $F_j \colon Q_1 \times \mathcal{S}(n) \to \mathbb{R}$ and a sequence of functions $f_j$ linked through
\begin{equation} \label{proof comp lemma eq01}
    F_j(X, D^2u_j) = f_j, \text{ in } Q_1,
\end{equation}
in the viscosity sense, with
\begin{equation} \label{proof comp lemma eq02}
    \intav{Q_1} \beta_j(X)^n dX + \int_{Q_1} |f(X)|^{n-\varepsilon} dX = \text{o}(1), \quad \text{ as } j \to \infty;
\end{equation}
however
\begin{equation} \label{proof comp lemma eq03}
    \sup\limits_{Q_{1/2}} |u_j - h| \ge \delta_0,
\end{equation}
for any $h$ satisfying \eqref{eq comp lemma} in $Q_{1/2}$. In \eqref{proof comp lemma eq02}, $\beta_j$ is the modulus of oscillation of the coefficients of the operator $F_j$ as in \eqref{beta}. By a consequence of Harnack inequality, see \cite{Es}, Lemma 2, we can assume, passing to a subsequence if necessary, that $u_j \to u_0$ locally uniformly in $Q_{11/20}$. Also, by uniform ellipticity, for each $X \in Q_1$ fixed, $F_j(X, M) \to F_0(X, M)$ locally uniformly in the space $\mathcal{S}(n)$. Finally, arguing as in the proof of \cite{C1}, Lemma 13, we conclude
$$
    F(0, D^2 u_0) = 0, \text{ in } Q_{1/2},
$$
in the viscosity sense and we reach a contradiction to \eqref{proof comp lemma eq03} for $j \gg 1$.
\end{proof}

\medskip


Our next lemma gives the first oscillation decay in our iterative scheme. We will write it in a bit more general form for future references.


\begin{lemma}\label{key lemma} Let $u$ be a viscosity solution to $F(X, D^2u) = f(X)$, with $|u| \le 1$,  in $Q_1$. Given $0 < \gamma < 1$, there exist $\eta > 0$ and $0 < \rho < 1/2$ depending only on $n$, $\Lambda,~\lambda$ and $\gamma$, such that if
\begin{equation}\label{Cond key lemma}
    \intav{Q_1} \beta(X)^n dX \le \eta^n \quad \text{and} \quad \int_{Q_1} |f(X)|^{n-\varepsilon} dX \le \eta^{n-\varepsilon},
\end{equation}
then, for some universally bounded constant $\mu \in \mathbb{R}$, $|\mu| \le C(n, \Lambda,\lambda)$, there holds
\begin{equation}\label{thesis key lemma}
    \sup\limits_{Q_{\rho}} |u - \mu | \le \rho^{\gamma}.
\end{equation}
\end{lemma}
\begin{proof} For a $\delta > 0$ to be chosen, we apply Lemma \ref{compactness} and find a function $h \colon Q_{1/2} \to \mathbb{R}$ satisfying
$$
    \mathfrak{F}(D^2 h) = 0, \text{ in } Q_{1/2},
$$
in the viscosity sense, such that
\begin{equation}\label{proof key lemma eq 01}
    \sup\limits_{Q_{1/2}} |u - h | \le \delta.
\end{equation}
From universal $C^{1,{\alpha}}$ regularity theory for viscosity solution to constant coefficient equations, (see, for instance, \cite{CC}, Corollary 5.7) there exist universal constants $\bar{\alpha} > 0$ and $C > 0$ such that
$$
    \|h\|_{C^{1,\bar{\alpha}} (Q_{1/3})} \le C.
$$
In particular,
\begin{equation}\label{proof key lemma eq 02}
    \sup\limits_{Q_{r}} |h - h(0) | \le Cr, \quad \forall r < 1/4.
\end{equation}
For $0< \gamma < 1$ given and fixed, we make the following universal selections
\begin{equation}\label{proof key lemma eq 03}
    \rho := \sqrt[1-\gamma]{\dfrac{1}{2C}} \quad \text{and} \quad \delta := \dfrac{1}{2} \rho^\gamma.
\end{equation}
The choice of $\delta$ determines $\eta$ through the Compactness Lemma \ref{compactness}. To finish up, we take $\mu = h(0)$ and, by triangular inequality,
$$
    \sup\limits_{Q_\rho} |u - \mu| \le \sup\limits_{Q_\rho} |h - \mu| + \sup\limits_{Q_\rho} |u - h| \le \rho^\gamma,
$$
as desired.
\end{proof}

We are ready to state and prove the main Theorem of this section,
which provides sharp modulus of continuity for solutions to
equations when the source function $f$ lies in the critical
integrability space $L^{n-\varepsilon}$.

\begin{theorem} \label{Calpha theory} Let $u$ be a viscosity solution to $F(X, D^2u) = f(X)$ in $Q_1$. There exists a universal constant $\vartheta_0 > 0$ such that if $\sup\limits_{Y\in Q_{1/2}} \|\beta(Y, \cdot )\|_{L^n} \le \vartheta_0$, then, for a universal constant $C> 0$, there holds
$$
    \|u\|_{C^{0, \frac{n-2\varepsilon}{n-\varepsilon}}(Q_{1/2})} \le C \left \{\|u\|_{L^\infty(Q_1)} + \|f\|_{L^{n-\varepsilon}(Q_1)} \right \},
$$
where $\varepsilon$ is the universal Escauriaza constant.
\end{theorem}

\begin{proof} Initially, by scaling and normalization, we can assume $|u| \le 1$ and $\|f\|_{L^{n-\varepsilon}(Q_1)} \le \eta$, where $\eta$ is the universal constant from Lemma \ref{key lemma} when $\gamma$ is taken to be $\frac{n-2\varepsilon}{n-\varepsilon}$.  Here we also set
$\vartheta_0 = \eta$, for the same choice above described. The strategy of the proof is to iterate Lemma \ref{key lemma} in dyadic cubes. More precisely, fixed $Y \in Q_{1/2}$, we will show that there exists a convergent sequence of real numbers $\mu_k$, such that
\begin{equation}\label{proof Calpha eq 01}
    \sup\limits_{Q_{\rho^k}(Y)} | u - \mu_k| \le \rho^{k \cdot \frac{n-2\varepsilon}{n-\varepsilon}},
\end{equation}
where $\rho$ is the radius granted from Lemma \ref{key lemma}, when taken $\gamma =  \frac{n-2\varepsilon}{n-\varepsilon}$. We prove \eqref{proof Calpha eq 01} by induction process. The step $k=1$ is precisely the contents of Lemma \ref{key lemma}. Suppose verified the $k$th step of  \eqref{proof Calpha eq 01}, we define
$$
    v_k(X) := \dfrac{u(Y+ \rho^k X) - \mu_k}{\rho^{k \cdot \frac{n-2\varepsilon}{n-\varepsilon}}}.
$$
We also define
\begin{equation}\label{proof Calpha eq 01.5}
    F_k(X, M) :=  \rho^{k \cdot [2 - \frac{n-2\varepsilon}{n-\varepsilon}]} F(X, \dfrac{1}{\rho^{k \cdot [2 - \frac{n-2\varepsilon}{n-\varepsilon}]}} M).
\end{equation}
Easily one verifies that $F_k$ is too $(\lambda, \Lambda)$-elliptic and
$$
    F_k(Y+\rho^k X, D^2v_k) = \rho^{k \cdot [2 - \frac{n-2\varepsilon}{n-\varepsilon}]} f(Y+ \rho^k X) =: f_k(X)
$$
in the viscosity sense. Furthermore,
$$
    \int_{Q_1} |f_k(X)|^{n-\varepsilon} dX \le \int_{Q_1} |f (X)|^{n-\varepsilon} dX \le \eta^{n-\varepsilon}.
$$
By the induction thesis \eqref{proof Calpha eq 01}, it follows that $|v_k| \le 1$. Thus, $v_k$ is entitled to Lemma \ref{key lemma}, which assures the existence of a universally bounded real number $\tilde{\mu}_k$, $|\tilde{\mu}_k|\le C$, such that
\begin{equation}\label{proof Calpha eq 02}
    \sup\limits_{Q_{\rho}} |v_k - \tilde{\mu} | \le \rho^{\frac{n-2\varepsilon}{n-\varepsilon}}.
\end{equation}
Define,
\begin{equation}\label{proof Calpha eq 03}
    \mu_{k+1} := \mu_k + \rho^{k \cdot \frac{n-2\varepsilon}{n-\varepsilon}} \tilde{\mu}_k.
\end{equation}
Rescaling \eqref{proof Calpha eq 02} to the unit picture gives the $(k+1)$th induction step in \eqref{proof Calpha eq 01}. It also follows from universal bounds that
\begin{equation}\label{proof Calpha eq 04}
    |\mu_{k+1} - \mu_k| \le C \rho^{k \cdot \frac{n-2\varepsilon}{n-\varepsilon}},
\end{equation}
thus the sequence is convergent and from \eqref{proof Calpha eq 01}, $\mu_k \to u(Y)$. Also from \eqref{proof Calpha eq 04}, it follows that
\begin{equation}\label{proof Calpha eq 05}
    |u(Y) - \mu_k| \le \dfrac{C}{1-  \rho^{\frac{n-2\varepsilon}{n-\varepsilon}}} \cdot \rho^{k \cdot \frac{n-2\varepsilon}{n-\varepsilon}}.
\end{equation}
Finally, given $0 < r < \rho$, let $k$ be the integer such that $\rho^{k+1} < r \le \rho^k$. We estimate from \eqref{proof Calpha eq 01} and \eqref{proof Calpha eq 05},
$$
    \begin{array}{lll}
        \displaystyle \sup\limits_{X \in Q_r(Y)} | u(X) - u(Y)| &\le&  \displaystyle \sup\limits_{X \in Q_{\rho^k} (Y)} | u(X) - \mu_k| + |u(Y) -               \mu_k| \\
        &\le& \left ( 1+ \dfrac{C}{1-  \rho^{\frac{n-2\varepsilon}{n-\varepsilon}}} \right ) \rho^{k \cdot \frac{n-2\varepsilon}{n-\varepsilon}} \\
        & \le & \dfrac{1}{\rho} \left ( 1+ \dfrac{C}{1-  \rho^{\frac{n-2\varepsilon}{n-\varepsilon}}} \right ) r^{\frac{n-2\varepsilon}{n-\varepsilon}}
    \end{array}
$$
and the proof of Theorem \ref{Calpha theory} is concluded.
\end{proof}

\bigskip

We finish up this section with few remarks on Theorem \ref{Calpha theory} as well as on some consequences of its proof.

\begin{remark}  \label{r1}
It is interesting to notice that our approach in fact gives
pointwise estimates on $u$. That is, if $\beta(X_0, X)$ is small
enough for some $X_0 \in Q_1$ fixed, then $u$ is $C^{0,
{\frac{n-2\varepsilon}{n-\varepsilon}}}$ continuous at $X_0$.
\end{remark}
\medskip
\begin{remark}\label{r3}
If $ n-\varepsilon \le q < n$ and $f \in L^q(Q_1)$, a simple adjustment
on the proof of Theorem \ref{Calpha theory} gives that $u \in
C^{0, \frac{2q - n}{q}}$ and
\begin{equation}\label{Calpha gen}
    \|u\|_{C^{0, \frac{2q - n}{q}}(Q_{1/2})} \le C \left \{\|u\|_{L^\infty(Q_1)} + \|f\|_{L^{q}(Q_1)} \right \},
\end{equation}
for a constant $C$ that depends only on $n$, $\Lambda, ~\lambda$ and
$q$. Such a result is in accordance, via embedding theorem, to the
regularity theory for equations with $W^{2,p}$ \textit{a priori}
estimates. For fully nonlinear equation, such estimates are only
available under the strong hypothesis of $C^{1,1}$ \textit{a
priori} estimates for $F$. In connection to Swiech  \textit{almost
optimal}  $W^{1,p}$ estimate,  \eqref{Calpha gen} represents the
prime sharp one.
\end{remark}
\medskip
\begin{remark}\label{r2}
Smallness condition on $L^n$-norm of $\beta(X_0, X)$ is obtained,
after scaling, when we assume the coefficients are in VMO. Without
any smallness assumption on the coefficients, Krylov-Safonov
Harnack inequality assures that $F$-harmonic functions, i.e.,
solutions to $F(X, D^2h) = 0$, are $C^{\beta}$ H\"older
continuous, for a universal $\beta$. The same reasoning employed
in this section gives that solutions to $F(X, D^2u) = f \in L^q$,
$q \ge n-\varepsilon$, lies in $C^{\min\{\frac{2q - n}{q},
\beta^{-}\}}$.
\end{remark}
\medskip
\begin{remark}\label{r4} Our approach explores only the
behavior of the distributional function of $f$; therefore, we can
replace, with no change in the reasoning, the $L^q$ norm of $f$ by
the weak-$L^q$ norm of $f$, or conditions like $\int_{B_r} |f| dX
\lesssim r^{\frac{n}{q'}}$.

\end{remark}

\medskip
\begin{remark}\label{r5}
It follows from  Remark \ref{r3} that if $f \in L^n(Q_1)$, then $u
\in C^{0, \alpha}_\text{loc}$ for any $\alpha < 1$, see also
\cite{S}, Corollary 2.2. The optimal universal modulus of
continuity $\omega(t)$, for the conformal threshold case should,
therefore, be of order $\text{o}(t^\alpha)$, for any $\alpha < 1$;
however, $\frac{w(t)}{t} \nearrow +\infty$ as $t = \text{o}(1)$.
The precise asymptotic behavior of $\omega$ is the objective of
Section \ref{Section Log-Lip}.
\end{remark}

\section{Log-Lipschitz estimates} \label{Section Log-Lip}

In this section we address the question of finding the optimal and universal modulus of continuity for solutions of uniformly elliptic equations with right-hand-side in the critical space $L^n$. Such estimate is particularly important to the general theory of fully nonlinear elliptic equations. Through a simple analysis one verifies that solution $u$ to $F(D^2u) = f \in L^n$ may not be Lipschitz continuous, though from Remark \ref{r5} $u$ is $C^{0,\alpha}$ for any $\alpha < 1$. Our ultimate goal is to prove that $u$ has a universal Log-Lipschitz modulus of continuity.

Our strategy here is also based on a fine compactness approach. Next lemma is pivotal for our iterative analysis.

\begin{lemma}\label{key lemma Log-Lip} Let $u$ be a viscosity solution to $F(X, D^2u) = f(X)$, with $|u| \le 1$,  in $Q_1$. There exist $\eta > 0$ and $0 < \rho < 1/2$ depending only on $n$, $\Lambda, ~ \lambda$ and $\gamma$, such that if
\begin{equation}\label{Cond key lemma}
    \intav{Q_1} \beta(X)^n dX \le \eta^n \quad \text{and} \quad \int_{Q_1} |f(X)|^{n} dX \le \eta^n,
\end{equation}
then, we can find an affine function $\ell(X) := a + \mathbf{b}\cdot X$, with universally bounded coefficients, $|a| + |\mathbf{b}| \le C(n, \Lambda,\lambda)$, such that
\begin{equation}\label{thesis key lemma}
    \sup\limits_{Q_{\rho}} |u(X) - \ell(X)| \le \rho.
\end{equation}
\end{lemma}
\begin{proof} The proof goes in the lines of the proof of Lemma \ref{key lemma}. For a $\delta > 0$ to be chosen, we apply Lemma \ref{compactness} and find a function $h \colon Q_{1/2} \to \mathbb{R}$ satisfying
$\mathfrak{F}(D^2 h) = 0, \text{ in } Q_{1/2},$ such that
\begin{equation}\label{proof key lemma eq 01}
    \sup\limits_{Q_{1/2}} |u - h | \le \delta.
\end{equation}
Define
\begin{equation}\label{def ell}
    \ell(X) := h(0) + \nabla h(0) \cdot X.
\end{equation}
From regularity theory available for $h$, there exist universal constants $\bar{\alpha} > 0$ and $C > 0$ such that
\begin{equation}\label{proof key lemma eq 02}
    \sup\limits_{Q_{r}} |h(X) - \ell(X)  | \le Cr^{1+ \bar{\alpha}}, \quad \forall r < 1/4.
\end{equation}
Finally, we  make the universal choices:
\begin{equation}\label{proof key lemma eq 03}
    \rho := \sqrt[\bar{\alpha}]{\dfrac{1}{2C}} \quad \text{and} \quad \delta := \dfrac{1}{2} \rho.
\end{equation}
Again, the choice of $\delta$ determines $\eta$ through the Compactness Lemma \ref{compactness}, which once more is a universal choice. The proof now ends as in the proof of Lemma \ref{key lemma}.
\end{proof}

We are ready to state and prove the optimal Log-Lipschitz
regularity estimate for solutions to equation with $L^n$
right-hand-sides.

\begin{theorem} \label{Log-Lip theory} Let $u$ be a viscosity solution to $F(X, D^2u) = f(X)$ in $Q_1$. There exists a universal constant $\vartheta_0 > 0$ such that if $\sup\limits_{Y\in Q_{1/2}} \|\beta(Y, \cdot )\|_{L^n} \le \vartheta_0$, then, for a universal constant $C> 0$, for any $X, Y \in Q_{1/2}$
$$
    |u(X) - u(Y)| \le C \left \{\|u\|_{L^\infty(Q_1)} + \|f\|_{L^{n}(Q_1)} \right \} \cdot \omega (|X-Y|),
$$
where
$$
    \omega(t) := -t \log t.
$$
\end{theorem}

\begin{proof} Again, by scaling and normalization process, we can assume $|u| \le 1$ and $\|f\|_{L^{n}(Q_1)} \le \eta$, where $\eta$ is the universal constant from Lemma \ref{key lemma Log-Lip}. Set $\vartheta_0 = \eta$. Fixed $Y \in Q_{1/2}$, we will show that there exists a sequence of affine functions
$$
    \ell_k(X) := a_k + \mathbf{b}_k \cdot (X-Y),
$$
such that
\begin{equation}\label{proof LL eq 01}
    \sup\limits_{Q_\rho^k(Y)} | u(X) - \ell_k(X)| \le \rho^k.
\end{equation}
where $\rho$ is the radius provided in Lemma \ref{key lemma Log-Lip}.  As before, we verify \eqref{proof LL eq 01} by induction. Lemma \ref{key lemma Log-Lip} gives the 1st step of induction. Suppose we have verified the $k$th step of \eqref{proof LL eq 01}, define
$$
    v_k(X) := \dfrac{(u - \ell_k)(Y+ \rho^k X)}{\rho^k}
$$
and
\begin{equation}\label{proof Calpha eq 01.5}
    F_k(X, M) :=  \rho^{k} F(X, \dfrac{1}{\rho^{k}} M).
\end{equation}
Again $F_k$ is $(\lambda, \Lambda)$-elliptic and
$$
    F_k(Y+\rho^k X, D^2v_k) = \rho^{k} f(Y+ \rho^k X) =: f_k(X)
$$
in the viscosity sense. Furthermore,
$$
    \int_{Q_1} |f_k(X)|^n dX \le \int_{Q_1} |f (X)|^n dX \le \eta^n.
$$
Thus, we can apply Lemma \ref{key lemma Log-Lip} to the scaled function $v_k$, which assures the existence of an affine function $\tilde{\ell}_k(X) = \tilde{a}_k + \tilde{\mathbf{b}}_k \cdot X$, with $|\tilde{a}_k| + |\tilde{\mathbf{b}}_k| \le C$, such that
\begin{equation}\label{proof LL eq 02}
    \sup\limits_{Q_{\rho}} |v_k - \tilde{\ell}_k | \le \rho.
\end{equation}
Define,
\begin{equation}\label{proof LL eq 03}
    a_{k+1} := a_k + \rho^{k} \tilde{a}_k, \quad \text{and} \quad \mathbf{b}_{k+1} := \mathbf{b}_k + \tilde{\mathbf{b}}_k.
\end{equation}
Rescaling \eqref{proof LL eq 02} to the unit picture gives the $(k+1)$th induction step in \eqref{proof LL eq 01}. The universal bounds from Lemma \ref{key lemma Log-Lip} yield
\begin{equation}\label{proof LL eq 04}
    |a_{k+1} - a_k| \le C \rho^{k} \quad \text{and} \quad |\mathbf{b}_{k+1} - \mathbf{b}_k| \le C.
\end{equation}
From \eqref{proof LL eq 01} and \eqref{proof LL eq 04}, we conclude the sequence $\{a_k\}_{k\ge 1}$ converges to $u(Y)$ and
\begin{equation}\label{proof LL eq 05}
    |u(Y) - a_k| \le \dfrac{C\rho^k}{1-  \rho}.
\end{equation}
The vector sequence $\{\mathbf{b}_k\}_{k\ge 1}$ may not converge; nevertheless we estimate
\begin{equation}\label{proof LL eq 05.5}
    |\mathbf{b}_k| \le  \sum\limits_{j=1}^{k} |\mathbf{b}_k - \mathbf{b}_{k-1}| \le C k.
\end{equation}
Given $0 < r < \rho$, let $k$ be the first integer such that $\rho^{k+1} < r$. We estimate from \eqref{proof LL eq 01}, \eqref{proof LL eq 05} and \eqref{proof LL eq 05.5}
$$
    \begin{array}{lll}
        \displaystyle \sup\limits_{X \in Q_r(Y)} | u(X) - u(Y)| &\le&  \displaystyle \sup\limits_{X \in Q_{\rho^k} (Y)} | u(X) - \ell_k| +                      |u(Y) -a_k| + |\mathbf{b}_k| \rho^k \\
        &\le& C( \rho^k + k\rho^k ) \\
        & \le & \dfrac{C}{\rho} \left ( r + \dfrac{\log r}{\log \rho}  \cdot r \right ) \\
        & \le & -C r \log r,
    \end{array}
$$
and we finish up the proof of Theorem \ref{Log-Lip theory}.
\end{proof}

\bigskip

As in Remark \ref{r4}, our approach explores only the behavior of
the distributional function of $f$, thus the same conclusion holds
if $f \in L^n_\text{weak}$, or if $f$ satisfies $\int_{B_r} |f| dX
\lesssim r^{n-1}$.
\par
We also point out that it is possible to adjust the proof
presented here as to obtain $\nabla u \in \text{BMO}(Q_{1/2})$,
provided $f \in L^n(Q_1)$. Indeed, under appropriate smallness
condition on $\|f\|_{L^n(Q_1)}$ and on $\| \beta\|_{L^n}$, it follows from
Swiech's $W^{1,q}$ interior estimates that one can find $h$, solution to
$F(X,D^2h) = 0$, $\varepsilon$-close to $u$ in the
$W^{1,q}(Q_{1/2})$ topology. Thus, in the proof of Lemma \ref{key
lemma Log-Lip}, in addition to \eqref{key lemma Log-Lip} one
obtains

\begin{equation}\label{grad bmo stp 1}
    \intav{Q_\rho} |\nabla (u - \ell)|^q dX \le 1.
\end{equation}
An iterative argument as in the proof of Theorem \ref{Log-Lip
theory} provides BMO estimate of the gradient in terms of the
$L^n$-norm of $f$:
\begin{equation}\label{bmo est grad}
    \| \nabla u \|_{\text{BMO}(Q_1/2)} \le C \left \{
    \|f\|_{L^n(Q_{1})} + \|u\|_{L^\infty(Q_1)} \right \}, \quad
    \text{for } r\le 1/4,
\end{equation}
Under $\text{VMO}$ assumption on the coefficients of $F$, a slightly finer scaling argument in fact grants that $\nabla u$
lies in the space of vanishing mean oscillation,
$\text{VMO}(Q_{1/2})$ and, for some modulus of continuity
$\omega$,
$$
    \intav{Q_r} |\nabla u(X) - (\nabla u)_r|^q dX = \omega\left
    ( \|f\|_{L^n(Q_r)} \right ),
$$
where, as usual, $(\nabla u)_r =\intav{Q_r}{\nabla u(X) dX}$. Such
an estimate is based on the same iterative argument developed
here, but with \eqref{grad bmo stp 1} replaced by $\intav{Q_\rho}
|\nabla (u - \ell)|^q dX \le  \tilde{\omega}(\int_{Q_1} |f(X)|^n
dX)$, where the modulus of continuity $\tilde{\omega}$ is given
indirectly by the compactness method.

\section{$C^{1,\nu}$ interior regularity}

In this intermediary Section, we comment on the $C^{1,\nu}$
interior regularity estimates obtained by an adjustment on the
approach from Section \ref{Section Log-Lip}. The results from this
Section are enclosed in previous works, \cite{C1}, \cite{S}. We
mention such estimates and their connections here for didactic and
completeness purposes.

\par

Initially, it does follow from the theory developed in \cite{C1}
that when $f \in L^q(Q_1)$, for $q > n$, viscosity solutions to
$F(X, D^2u) = f(X)$ are differentiable and indeed they lie in
$C^{1,\mu}_\text{loc}$ for some $\mu$. As a consequence of our
analysis from Section \ref{Section Log-Lip}, we obtain a simple
proof of the fact that the optimal regularity estimate available
for the equation $F(X, D^2u) = f \in L^q$, $q > n$, is
$C^{1,\nu}_\text{loc}$, where the critical exponent $\mu$ is given
by
\begin{equation}\label{optimal C1alpha theory}
    \nu := \min \{ \dfrac{q-n}{q}, \bar{\alpha}^{-} \},
\end{equation}
and $\bar{\alpha}$ is the universal optimal exponent from the
$C^{1,\bar{\alpha}}$ regularity theory for solutions to
homogeneous $(\lambda, \Lambda)$-elliptic operators with constant
coefficients. The sharp relation stated in \eqref{optimal C1alpha
theory} should be read as follows:

\begin{equation} \label{explain sharp}
\left \{
\begin{array}{llll}
     \text{ If} & \frac{q-n}{q} < \bar{\alpha} & \text{ then } & u \in C_\text{loc}^{1, \frac{q-n}{q}}.  \\
     \text{ If} &\frac{q-n}{q} \ge \bar{\alpha} &  \text{ then } & u \in C^{1, \beta}_\text{loc},  \text{ for any } \beta < \bar{\alpha}.
\end{array}
\right.
\end{equation}

Such a regularity estimate is in fact the best possible under such
a general condition on $f$ and it agrees with the result obtained
in \cite{S}, Theorem 2.1, case 2. It is also interesting to
compare \eqref{optimal C1alpha theory} with \cite{C1}, Theorem 2.
In particular, if $f \in L^\infty$, then viscosity solutions to
$F(X, D^2u) = f \in L^\infty$ with VMO coefficients are
\textit{almost} as smooth as $F$-harmonic functions, $F(D^2h) =
0$, up to $C^{1,1^{-}}$. Optimal estimates under slightly stronger
assumption on $F$ will be addressed in Section \ref{section
C1,LL}.

To verify such a regularity estimate from our strategy we revisit
Lemma \ref{key lemma Log-Lip} and check that, for any $\alpha <
\bar{\alpha}$, under the smallness assumption on $\intav{Q_1}
\beta^n$ and on $\|f\|_{L^q}$, it is possible to choose $\rho =
\rho(\alpha)$ such that
$$
    \sup\limits_{Q_{\rho}} |u(X) - \ell(X)| \le \rho^{1+\alpha},
$$
where $\ell$ is as entitled in \eqref{def ell}. For $\nu$ as in
\eqref{optimal C1alpha theory} the re-scaled function
$$
    v_1(X) := \dfrac{(u-\ell)(\rho X)}{\rho^{1+\nu}},
$$
satisfies $\tilde{F}(X, D^2v) = \rho^{1-\nu} f(\rho X) := \tilde{f}(X)$ for some $(\lambda, \Lambda)$-elliptic operator
$\tilde{F}$ and $\|\tilde{f}\|_{L^q} \le  \| f\|_{L^q}$. By iteration, we produce a sequence of affine functions
$$
    \ell_k(X) \longrightarrow u(Y) + \nabla u(Y) \cdot X,
$$
for any fixed $Y \in Q_{1/5}$. Ultimately we conclude
$$
\sup\limits_{X \in Q_r(Y)} |u(X) - (u(Y) + \nabla u(Y) \cdot X)|
\le C r^{1+\nu},
$$
as desired.

\section{Optimal $C^{1,\mathrm{Log-Lip}}$ regularity} \label{section
C1,LL}

In this final section we close the regularity theory for fully
nonlinear equations, $F(X, D^2u) = f(X)$, by addressing the
optimal regularity estimate available for the limiting upper
borderline case  $f\in L^\infty$ -- or better yet $f \in
\text{BMO}$. It fact, it has become clear through the theory of
Harmonic Analysis and its applications that Jonh-Nirenberg's space
of bounded mean oscillation is the correct endpoint for $L^p$
integrability condition as $p \nearrow +\infty$.
    \par
For simplicity, we will only work on constant coefficient
equations. As before, similar result can be shown under
appropriate continuity assumption on the coefficients -- $C^{0,
\bar{\epsilon}}$ is enough. In view of the \textit{almost} optimal
estimate given in \eqref{optimal C1alpha theory}, if we want to
establish a fine regularity theory for solutions $F(D^2u) = f \in
L^\infty$, it is natural to assume that $F$-harmonic functions are
$C^{2+\epsilon}$ smooth; otherwise no further information could be
reveled from better hypotheses on the source function $f$. We now
state and prove our sharp $C^{1, \mathrm{Log-Lip}}$ interior
regularity theorem.

\begin{theorem}\label{C1,Log-Lip theory with BMO} Let $u$ be a viscosity solution to $F(D^2u) = f(X)$ in $Q_1$. Assume that
 for any matrix $M \in \mathcal{S}(n)$, with $F(M) = 0$, solutions to $F(D^2 h + M) = 0$, satisfies
 \begin{equation}\label{C11 hyp}
    \|h\|_{C^{2,\epsilon}(Q_{r})} \le \Theta r^{-{(2+\epsilon)}} \|h\|_{L^\infty(Q_1)},
 \end{equation}
 for some $\Theta >0$.  Then, for a constant $C> 0$, depending only on $\Theta, ~\epsilon$ and universal parameters, there holds
\begin{equation}\label{thesis C1,LL}
    \left | u(X) - \left [ u(0) + \nabla u(0) \cdot X \right ] \right| \le C \left \{\|u\|_{L^\infty(Q_1)} + \|f\|_{\mathrm{BMO}(Q_1)} \right \} \cdot |X|^2 \log |X|^{-1}
\end{equation}
\end{theorem}

\begin{proof} As before, we assume, with no loss of generality, that $\|u\|_{L^\infty(Q_1)} \le 1$ and
$\|f\|_{\text{BMO}(Q_1)} \le \vartheta_0$ for some $\vartheta_0>
0$ to be determined. Hereafter in this proof, we shall label
\begin{equation} \label{avg f}
    \langle f \rangle := \intav{Q_1} f(Y) dY.
\end{equation}
Notice that universal bound for $\|f\|_{\text{BMO}(Q_1)}$ implies
$|\langle f \rangle|$ is under control. From hypothesis \eqref{C11
hyp}, equation $F(D^2h + M) = c$ also has $C^{2, \epsilon}$
interior estimates with constant $\tilde{\Theta}$, depending on
$\Theta$ and $|c|$, for any matrix $M$ satisfying $F(M) = c$. Our
strategy now is to show the existence of a sequence of quadratic
polynomials
$$
    \mathfrak{P}_k(X) := a_k + \mathbf{b}_k \cdot X + \dfrac{1}{2}X^t M_k X,
$$
where, $\mathfrak{P}_0 = \mathfrak{P}_{-1}  = \frac{1}{2} \langle
QX,X \rangle$, where $F(Q) = \langle f \rangle$, and for all $k\ge
0$,
\begin{eqnarray}
    \label{P1} F(M_k) = \langle f \rangle, &\\
    \label{P2} \sup\limits_{Q_{\rho^k}} |u - \mathfrak{P}_k| \le  \rho^{2k}, & \\
    \label{P3} |a_k - a_{k-1}| + \rho^{k-1} |\mathbf{b}_k - \mathbf{b}_{k-1}| +  \rho^{2(k-1)}  |M_k - M_{k-1}|
    \le C \rho^{2(k-1)}. &
\end{eqnarray}
The radius $\rho$ in \eqref{P2} and \eqref{P3} is determined as
\begin{equation}\label{rho}
    \rho := \sqrt[\epsilon]{\dfrac{10}{\Theta}} < 1/2.
\end{equation}
We shall verify that by induction. The first step $k=0$ is
immediately satisfied. Suppose we have verified the thesis of
induction for $k=0,1, \cdots, i$. Define the re-scaled function $v
\colon Q_1 \to \mathbb{R}$ by
$$
    v(X) := \dfrac{(u - \mathfrak{P}_i)(\rho^i X)}{\rho^{2i}}.
$$
Immediately one verifies that $v$ solves
$$
    F( D^2v + M_i) = f(\rho^{i} X) =: f_i(X)
$$
As in  Lemma \ref{compactness}, under smallness assumption on
$$
    [f]_{\text{BMO}} = \sup\limits_{0<r\le 1} \int_{Q_r} |f(X) - \intav{Q_r}{f(Y)dY}
    |^n dX,
$$
to be regulated soon, we find $h$, solution to
 \begin{equation}\label{eq for h}
     F(D^2h +M_i) = \langle f \rangle,
\end{equation}
that stays within a $\delta$-distance in the $L^\infty$-topology
to $v$ in $Q_{1/2}$.  From  hypothesis \eqref{C11 hyp}, $h$ is
$C^{2,\epsilon}$ at the origin with universal bounds. If we define
 $$
    \tilde{\mathfrak{P}}_i(X) = h(0) + \nabla h(0) \cdot X + \dfrac{1}{2} X^t D^2h(0) X,
$$
by $C^{2,\epsilon}$ regularity assumption \eqref{C11 hyp}, we  bound
\begin{equation}\label{C1,LL eq01}
    |h(0)| + |\nabla h(0) | + |D^2h(0) | \le   C{\Theta}.
\end{equation}
Select  
$$
	\delta = \frac{4}{5} \rho^2,
$$
where $\rho$ is the number in \eqref{rho}.  These choices depend only on $\Theta$, $\epsilon$ and universal constants. Again the selection of $\delta$ determines, via Lemma \ref{compactness}, the smallness assumption given by the constant $\vartheta_0 > 0$. Triangular inequality and $C^{2,\epsilon}$ regularity theory for $h$ give,
\begin{equation}\label{C1,LL eq02}
    \sup\limits_{Q_\rho} |v -\tilde{\mathfrak{P}}_i| \le \rho^2.
\end{equation}
Rescaling \eqref{C1,LL eq02} back to the unit picture yields
\begin{equation}\label{C1,LL eq021}
    \sup\limits_{Q_{\rho^{i+1}}} \left |u(X)  -\left [\mathfrak{P}_i(X) + \rho^{2i} \tilde{\mathfrak{P}}_i(\rho^{-i} X) \right ] \right | \le \rho^{2(i+1)}.
\end{equation}
Therefore, defining
$$
    \mathfrak{P}_{i+1}(X) := \mathfrak{P}_i(X) + \rho^{2i} \tilde{\mathfrak{P}}_i(\rho^{-i} X),
$$
we verify the $(i+1)$th step of induction. Notice that condition
\eqref{P1} is granted by equation \eqref{eq for h}. From
\eqref{P3} we conclude that $a_k \to u(0)$ and $\mathbf{b}_k \to
\nabla u(0)$, in addition
\begin{eqnarray}
    \label{Conv 1} |u(0) - a_k| \le C \rho^{2k} \\
    \label{Conv 2} |\nabla u(0) - \mathbf{b}_k| \le C \rho^k.
\end{eqnarray}

Notice that from \eqref{P3} it is not possible to assure convergence of the matrices $M_k$; nevertheless, we estimate
\begin{eqnarray}
    \label{Conv 3}|M_k| \le C k.
\end{eqnarray}
Finally, given any $ 0 < r< 1/2$, let $k$ be the integer such that
$$
    \rho^{k+1} < r \le  \rho^k.
$$
We estimate, from \eqref{Conv 1}, \eqref{Conv 2} and \eqref{Conv 3},
$$
    \begin{array}{lll}
        \sup\limits_{Q_r} \left | u(X) - \left [ u(0) + \nabla u(0) \cdot X \right ] \right| &\le& \rho^{2k} + |u(0) - a_k| + \rho  |\nabla u(0) - \mathbf{b}_k| + \rho^{2k}|M_k| \\
        &\le & -C r^2 \log r,
    \end{array}
$$
and the Theorem is proven.
\end{proof}

\bigskip

Adjustments on the proof of Theorem \ref{C1,Log-Lip theory with
BMO}, similar to the ones explained at the end of Section
\ref{Section Log-Lip}, yield that $D^2u \in \text{BMO}(Q_{1/2})$,
with appropriate \textit{a priori} estimate on the
$\text{BMO}(Q_{1/2})$ norm of the Hessian of $u$ in terms of the
$\text{BMO}$ norm of $f$ in $Q_1$. Therefore, by the same slightly
finer scaling procedure mentioned at the end of Section
\ref{Section Log-Lip}, we conclude that, if $f \in \text{VMO}$,
then so does $D^2u$.

\par

As mentioned earlier, Theorem \ref{C1,Log-Lip theory with BMO}
holds true for equations with $C^{0, \bar{\epsilon}}$
coefficients, or more generally under the condition
$$
    \intav{Q_r}   | \tilde{\beta}_F(X)  |^n dX \le \eta^n \cdot  r^{n \bar{\epsilon}},
$$
for some $\eta > 0$ small enough, where, as in \cite{C1}, \cite{CC},
$$
    \tilde{\beta}_F(X) := \sup\limits_{N \in \mathcal{S}(n)}  \dfrac{\left |F(X,N) - F(0,N) \right |}{1+\|N\|}.
$$
Indeed, under $C^{0, \bar{\epsilon}}$ continuity of the coefficients and hypothesis on $C^{2,\epsilon}$ \textit{a priori} estimates for $F(0,D^2h) = 0$, it follows from \cite{C1} that solutions to  variable coefficients, homogeneous equation, $F(X, D^2\xi) = 0$ are $C^{2,\delta}$, with appropriate \textit{a priori} estimates.

Evans and Krylov regularity theory,  \cite{E} and \cite{K1, K2},
assures that convex or concave operators do satisfy \eqref{C11
hyp}. The work of Cabre and Caffarelli, \cite{CC03}, provides
another class of equations for which one can apply Theorem
\ref{C1,Log-Lip theory with BMO}. We further mention that Theorem
\ref{C1,Log-Lip theory with BMO} can easily be written as the
following local \textit{a priori} estimates
\begin{equation}\label{thesis C1,LL general}
    \sup\limits_{Q_r(Y)} \left | u(X) - \left [ u(Y) + \nabla u(Y) \cdot X \right ] \right| \le C \left \{
    \|u\|_{L^\infty(Q_1)} + \|f\|_{\text{BMO}(Q_1)} \right \}  r^2 \log r^{-1},
\end{equation}
for any $ Y \in Q_{1/2}$, $ r < 1/2$. After these comments, we can state the following Corollary.

\begin{corollary} \label{cor theo C1,LL} Let $F \colon Q_1 \times \mathcal{S}(n)$ be a $(\lambda, \Lambda)$-elliptic operator. Define
$$
    \tilde{\beta}_F(X,Y) :=   \sup\limits_{N \in \mathcal{S}(n)}  \frac{\left |F(X,N) - F(Y,N) \right |}{1+\|N\|}
$$
and assume $\sup\limits_{Y \in Q_1} \| \tilde{\beta}_F(X,Y)
\|_{C^{0,\epsilon}} :=K < \infty$. Let $u$ be a solution to $F(X,
D^2u) = f \in  \mathrm{BMO}(Q_1)$. Then, for some $C > 0$ that
depends only on $C^{2,\epsilon}$ \textit{a priori} regularity
estimates available for $F(X, D^2u) = 0$, $K$, and universal
constants, there holds
\begin{equation}\label{thesis C1,LL MORE general}
    \sup\limits_{Q_r(Y)} \left | u(X) - \left [ u(Y) + \nabla u(Y) \cdot X \right ] \right| \le C \left \{
    \|u\|_{L^\infty(Q_1)} + \|f\|_{\mathrm{BMO}(Q_1)} \right \}  r^2 \log r^{-1}.
\end{equation}
\end{corollary}

\medskip

It is also  simple to verify that the thesis of Corollary \ref{cor
theo C1,LL}, estimate \eqref{thesis C1,LL MORE general},  implies,
for an \textit{a priori} $C^{1,1}$ solution $u$, the corresponding
bounds to the adimensional $C^{1, \mathrm{Log-Lip}}$ norm
$$
    \|u\|^{*}_{C^{1, \mathrm{Log-Lip}}(Q_d)} := \|u\|_{L^\infty(Q_d)} + d \|\nabla u\|_{L^\infty(Q_d)}  + d^2\log d^{-1} \|D^2u\|_{L^\infty(Q_d)}.
$$

\bigskip

\bigskip

\noindent{\bf Acknowledgment.} Part of this paper was written
while the author was visiting University of Chicago and University
of Texas at Austin. The author would like to thank these
institutions for their warm hospitality. The author also would
like to thank Luis Silvestre, for insightful comments and
suggestions, in particular for pointing out improvements to
Theorem \ref{C1,Log-Lip theory with BMO}. This work is partially
supported by CNPq-Brazil.

\bibliographystyle{amsplain, amsalpha}

\begin{thebibliography}{60}

\bibitem[CC03]{CC03} Cabre, Xavier; Caffarelli, Luis A.  \textit{Interior $C^{2,\alpha}$ regularity theory for a class of nonconvex fully nonlinear elliptic equations.}    J. Math. Pures Appl.    \textbf{82} (9) (2003), 573--612

\bibitem[Caff89]{C1} Caffarelli, Luis A. \textit{Interior a priori estimates for solutions of fully nonlinear equations.} Ann.
of Math. (2) \textbf{130} (1989), no. 1, 189--213.

\bibitem[CC95]{CC} Caffarelli, Luis A.; Cabr\'e, Xavier
\textit{Fully nonlinear elliptic equations.} American Mathematical
Society Colloquium Publications, 43. American Mathematical Society,
Providence, RI, 1995.

\bibitem[Ev82]{E} Evans, L. C.,
\textit{Classical solutions of fully nonlinear, convex, second-order elliptic equations}. Comm. Pure Appl. Math.
35(3), 333--363, 1982.

\bibitem[Es93]{Es} Escauriaza, L. \textit{$W^{2,n}$ a priori estimates for solutions to fully non-linear elliptic equations}. Indiana Univ. Math. J. \textbf{42}, no. 2 (1993), 413--423.

\bibitem[K82]{K1}  Krylov, N. V. \textit{Boundedly nonhomogeneous elliptic and parabolic equations.} Izv. Akad. Nak. SSSR Ser. Mat. \textbf{46} (1982), 487--523; English transl. in Math USSR Izv. 20 (1983), 459--492.

\bibitem[K83]{K2} Krylov, N. V. \textit{Boundedly nonhomogeneous elliptic and parabolic equations in a domain.} Izv. Akad. Nak. SSSR Ser. Mat. \textbf{47} (1983), 75--108; English transl. in Math USSR Izv. 22 (1984), 67--97.

\bibitem[KS79]{KS1} Krylov, N. V.; Safonov, M. V. \textit{An estimate of the probability that a diffusion process hits a set of positive measure.} Dokl. Akad. Nauk. SSSR \textbf{245} (1979), 235--255. English translation in Soviet Math Dokl. 20 (1979), 235--255.

\bibitem[KS80]{KS2} Krylov, N. V.; Safonov, M. V. \textit{Certain properties of solutions of parabolic equations with measurable coefficients.} Izvestia Akad Nauk. SSSR \textbf{40} (1980), 161--175.


\bibitem[M]{M} U. Menne,
\textit{Decay estimates for the quadratic tilt-excess of integral varifolds}. Arch. Ration. Mech. Anal. (to appear)

\bibitem[NV07]{NV1} N. Nadirashvili and S. Vladut,
\textit{Nonclassical solutions of fully nonlinear elliptic equations}. Geom. Funct. Anal. {\bf 17} (2007), no. 4, 1283–1296.

\bibitem[NV08]{NV2} N. Nadirashvili and S. Vladut,
{\it Singular viscosity solutions to fully nonlinear elliptic
equations.} J. Math. Pures Appl. (9) {\bf 89} (2008), no. 2,
107--113.

\bibitem[NV]{NV}  N. Nadirashvili and S. Vladut,
\textit{Nonclassical Solutions of Fully Nonlinear Elliptic
Equations II. Hessian Equations and Octonions}. Geom. Funct. Anal.
{\bf 21} (2011), 483-498

\bibitem[S97]{S} A. Swiech, \textit{$W^{1,p}$-Interior estimates for solutions of fully nonlinear, uniformly elliptic equations}.  Adv. Differential Equations {\bf 2} (1997), no. 6, 1005--1027.


\end{thebibliography}

\bigskip

\noindent \textsc{Eduardo V. Teixeira} \\
\noindent Universidade Federal do Cear\'a \\
\noindent Departamento de Matem\'atica \\
\noindent Campus do Pici - Bloco 914, \\
\noindent Fortaleza, CE - Brazil 60.455-760 \\
 \noindent \texttt{eteixeira@ufc.br}

\end{document}